\documentclass[11pt,a4paper]{amsart}
\usepackage{amsfonts}
\usepackage{amssymb}
\usepackage{amsmath}
\usepackage{amsthm}
\usepackage{cancel}
\usepackage{color}
\usepackage{cite}
\usepackage{enumitem}
\usepackage{tikz}
\usepackage{subfigure}
\numberwithin{equation}{section}
\theoremstyle{plain}
\newtheorem{theorem}{Theorem}[section]
\newtheorem{thm}[theorem]{Theorem}
\newtheorem{remark}[theorem]{Remark}
\newtheorem{remarks}[theorem]{Remarks}
\newtheorem{defn}[theorem]{Definition}
\newtheorem{lem}[theorem]{Lemma}
\newtheorem{prop}[theorem]{Proposition}
\newtheorem{cor}[theorem]{Corollary}

\providecommand{\sm}{\setminus}
\providecommand{\N}{\mathbb{N}}
\providecommand{\R}{\mathbb{R}}

\providecommand{\eps}{\varepsilon}

\providecommand{\dx}{}

\DeclareMathOperator{\supp}{supp}

\DeclareMathOperator{\spa}{span}

\renewcommand{\qed}{\hfill $\Box$}

\setitemize{itemsep=+2pt}
\setenumerate{itemsep=+2pt} 
\def\dys{\displaystyle}
\allowdisplaybreaks

\begin{document}
\title[Oscillating solutions for nonlinear  Helmholtz Equations]{\sc 
Oscillating solutions for nonlinear  Helmholtz Equations}
\author[Rainer Mandel, Eugenio Montefusco and Benedetta Pellacci]
{Rainer Mandel$^{1}$, Eugenio Montefusco$^{2}$, Benedetta Pellacci$^3$}

\address[R. Mandel]{Institut f\"ur Analysis, Karlsruher Institut f\"ur Technologie, Englerstra{\ss}e 2, 76131
Karlsruhe}
\email{Rainer.Mandel@kit.edu}
\address[E. Montefusco]{Dipartimento di Matematica, 
''Sapienza'' Universit\`a di Roma, p.le Aldo Moro 5, 00185, Roma,
 Italy.}
\email{eugenio.montefusco@uniroma1.it}
\address[B. Pellacci]{Dipartimento di Scienze e Tecnologie, 
Universit\`a di Napoli ''Parthenope'', Centro Direzionale, Isola C4 80143 Napoli, Italy.}
\email{benedetta.pellacci@uniparthenope.it}

\date{}
 \subjclass{ 35J05, 35J20, 35Q55.}
\keywords{ Nonlinear Helmholtz equations, standing waves, oscillating solutions} 
\maketitle

\begin{abstract}
Existence results for radially symmetric oscillating solutions for a class of nonlinear autonomous
Helmholtz equations are given and their exact asymptotic behavior at infinity is established.
Some generalizations to nonautonomous radial equations as well as existence results for nonradial solutions
are found. Our theorems prove the existence of standing waves solutions of 
nonlinear Klein-Gordon or Schr\"odinger equations with large frequencies.
\end{abstract} 

\footnotetext[1]{Research partially 
supported by the German Research Foundation (DFG) through the grant
{MA~6290/2-1} and CRC 1173.} 
\footnotetext[3]{Research partially 
supported by  MIUR-PRIN project $2015KB9WPT_006$ - PE1,  ``Gruppo
Nazionale per l'Analisi Matematica, la Probabilit\`a e le loro 
Applicazioni'' (GNAMPA) of the Istituto Nazionale di Alta Matematica
(INdAM); University Project ''Sostegno alla ricerca individuale 
per il trienno 2016-2018''.}

\section{\sc  Introduction}

The main aim of this paper is to give existence results for the following class of nonlinear equations
\begin{equation}\label{eq:ge}
  -\Delta u  = g(u)\qquad \text{in }\R^N
\end{equation}
with $N\geq 1$ and assuming that  the nonlinearity $g$ is such that
\begin{align}
\label{g:reg} g\in C^{1,\sigma}(\R) &\;  \text{ for some $\sigma\in (0,1)$},
\\
\label{g:odd} \text{$g$ is odd,  }&
\\
\label{g:der} \;  g'(0)>0,&
\\
\label{g:sign} 
\exists \alpha_{0}\in (0,+\infty]\,:\;\text{$g$ is positive} & \text{ on $(0,\alpha_0)$ and negative
    on $(\alpha_0,\infty)$. }
\end{align}
There is a huge  literature concerning \eqref{eq:ge} and nonautonomous 
variants of it  under the assumption $g'(0)<0$. Two seminal papers in this 
context are the contributions by  Berestycki-Lions and Strauss 
\cite{beli, Strauss} who proved the existence of smooth radially 
symmetric and exponentially decaying
solutions for a large class of nonlinearities with this property. We refer to 
the monographs
\cite{AmMa_nonlinear_analysis, Willem} for more results in this context. 
One  of the main interests in finding solutions of \eqref{eq:ge} 
is motivated by the fact that a solution $u\in H^{1}(\R^{N})$ of 
\eqref{eq:ge}  gives rise to a standing wave, i.e. a solution  
of the form $\psi(x,t)=e^{i\lambda t}u(x)$,
of the nonlinear time-dependent  Klein-Gordon equation
$$
\dfrac{\partial^{2}\psi}{\partial t^{2}}-\Delta\psi+V_{0}\psi=f(\psi)
\qquad (t,x)\in \R\times \R^{N}
$$
with $f(z)=g(|z|)\tfrac{z}{|z|}$. Therefore, the assumption \eqref{g:der} amounts to look for standing waves having low
frequencies $\omega<V_0$ and numerous existence results for 
$H^1(\R^N)$-solutions under this assumption
can be found in the references mentioned above. In this paper we deal 
with nonlinearities satisfying $g'(0)>0$, which gives rise to standing 
waves with large frequencies $\omega>V_0$.
Looking at  the form of the linearized operator $-\Delta-g'(0)$,
one realizes that $u_{0}=0$ lies in its essential spectrum
 and we are actually dealing with a class of nonlinear Helmholtz
equations. Furthermore, as explained in subsection 2.2 in \cite{beli}, 
the hypothesis \eqref{g:der} has the striking consequence that radially 
symmetric $H^1(\R^{N})$ solutions of \eqref{eq:ge}  can not exist, and
usual variational methods fail.
On the other hand, \eqref{g:der} is naturally linked to \eqref{g:sign}; in particular,  if
$g(z)/z$ decreases in $(0,+\infty)$, then \eqref{g:der} turns out to be necessary in order to have $H^{1}(\R^{N})$
solutions. Actually, the relevant solutions naturally lie outside this functional space. This fact can also be illustrated
by an examining  the behaviour of the minimal energy solutions on a
sequence of large bounded domains. Namely, in Theorem \ref{teo:app} we will show that if one takes a sequence of 
bounded domains  $\Omega_{n}$ invading $\R^{N}$, then 
\eqref{g:der} guarantees the existence of a sequence $(u_{n})$
of global minimizers of the associated  action functional over $H_0^1(\Omega_n)$
for sufficiently large $n$. But, it results that 
$(u_{n})$  converges in
$C^2_{{\rm loc}}(\R^{N})$ to the constant solution 
$u\equiv \alpha_{0}$. 

Therefore, under the assumption \eqref{g:der}, one has to look for solutions in a
broader class of functions. Our focus will be on oscillating and localized ones which we define as follows.

\begin{defn}
A distributional solution
$u\in C^{1,\alpha}(\R^{N})$  of \eqref{eq:ge} is called oscillating if it has an unbounded sequence of zeros. It
is called localized when
it converges to zero at infinity.  
\end{defn} 

Let us notice that, the Strong Maximum Principle implies that
oscillating solutions of  \eqref{eq:ge} change sign at each of their
zeros; so that we are going to find solutions that change sign 
infinitely many times. 

In our study, we will pay particular attention to the following model cases 
\begin{align}
\label{model:g1}
g_1(z)=-\lambda z+\dfrac{z}{s+z^2}\;\;  &\text{ where } s>0,\,\lambda<\frac{1}{s}. \\
\label{model:g23} 
g_2(z) = k^2 z - |z|^{p-2}z, \; \qquad
&
g_3(z) = k^2 z + |z|^{p-2}z \; \text{for $k\neq 0$}. 
\end{align}
Our interest in these examples has various motivations.
The nonlinearity $g_{1}$ is related to the study of the propagation
of lights beams in a photorefractive crystals (see \cite{ChristPRL, YangNJP}) when a saturation effect is taken into account. 
Differently from the more frequently studied model  
$$
  \tilde{g}(z):=-\lambda z+\dfrac{z^{3} }{1+sz^{2}},
$$
see e.g. \cite{ChenPRE},  $g_{1}$ describes a transition 
from the linear propagation and the saturated one.
This difference has important consequences, for instance for 
$g=\tilde{g}$ there are $H^{1}(\R^{N})$ solutions of \eqref{eq:ge} (e.g. see Theorem 3.6 in \cite{StuZhou}),
whereas, as we have already observed, this is not the case if $g=g_{1}$ due to $g_1'(0)>0$. Notice that, as
$\lambda<1/s$, equation \eqref{eq:ge} for $g=g_{1}$ can be rewritten in the following form
\begin{equation}\label{eq:hel}
-\Delta u -k^{2} u =- \frac{u^{3}}{s(s+u^2)} \quad\text{in }\R^N
\quad \text{with $k^{2}=\frac1s-\lambda$}
\end{equation}
which allows to settle the problem in $H^{1}(\R^{N})$ in every dimension $N$ and that shows that also this saturable model is included in the
class of the nonlinear Helmholtz equations. The principal difference between \eqref{model:g23} and 
\eqref{model:g1} is that the formers are superlinear and homogeneous nonlinearities, while the latter
is not homogeneous and it is asymptotically linear. However, all of them satisfy our general assumptions, with 
$\alpha_{0}\in (0,+\infty)$ for $g_{1}$ and $g_{2}$, and $\alpha_{0}=+\infty$ for $g_{3}$.

\smallskip

Up to now  nonlinear Helmholtz equations \eqref{eq:ge} have been 
mainly investigated for the model nonlinearity
$g_3$ or more general superlinear nonlinearities, even not autonomous.
In a series of papers \cite{evwe, evwe_dual, ev_orlicz, evwe_branch}
Ev\'{e}quoz and Weth proved the existence of radial and nonradial real, 
localized solutions of this equation
under various different assumptions on the nonlinearity.
Let us mention that some of the tools used in \cite{evwe_dual} had 
already appeared in a paper by
Guti\'{e}rrez \cite{gut} where the existence of complex-valued solutions 
was proved for space dimensions
$N=3,4$.  Let us first focus our attention on  radially symmetric solutions and state our first result, which provides a complete
description of the radially symmetric solutions of \eqref{eq:ge}.

\begin{theorem}\label{th:nD}
Assume \eqref{g:reg},\eqref{g:odd},\eqref{g:der},\eqref{g:sign}.  
Then there is a continuum 
$\mathcal{C} = \{ u_\alpha \in C^2(\R^N): |\alpha|< \alpha_0\}$ in 
$C^2(\R^N)$ consisting of radially symmetric
 oscillating solutions of \eqref{eq:ge}  having the following 
properties for all $|\alpha|<\alpha_0$:
\begin{itemize}
\item[(i)] $u_\alpha(0)=\alpha$,
\item[(ii)] $\|u_\alpha\|_{L^{\infty}(\R^{N})}=|\alpha|, \|u_\alpha'\|_{L^{\infty}(\R^{N})} \leq 
\sqrt{2G(\alpha)}$.
\end{itemize}
Moreover, for $N=1$ all these solutions are periodic; whereas, 
for $N\geq 2$ they are localized and satisfy the following
asymptotic behavior:
\begin{itemize}
\item[(iii)] There are positive numbers $c_\alpha,C_\alpha>0$ such that 
$$
c_\alpha r^{(1-N)/2} \leq |u_\alpha(r)|+|u_\alpha'(r)|+|u_\alpha''(r)| 
\leq C_\alpha r^{(1-N)/2}
\quad\text{for all }r\geq 1.
$$
\end{itemize}
\end{theorem}     

Here a continuum in $C^2(\R^N)$ is a connected subset of 
$C^2(\R^N) $ with respect to the uniform
convergence of the zeroth, first and second derivatives. 
The continuum $ \mathcal{C}$ found in Theorem \ref{th:nD} is
even maximal in the sense that there are no further radially symmetric 
localized solutions as we will see
in section~\ref{sec:radial}. 
Moreover,  conclusion $(iii)$ states that  the property
 $u_\alpha\in L^s(\R^N)$ is equivalent
to $u_\alpha\in W^{2,s}(\R^N)$,
and this happens if and only if 
$s>\frac{2N}{N-1}$. Notice that this implies that $u_{\alpha}\notin L^{2}(\R^{N})$, showing again that the
solutions, as expected, live outside the commonly used energy space.  

Furthermore, let us stress that the behaviour of the nonlinearity 
beyond $\alpha_0$ is completely irrelevant, in particular, the negativity of $g$ on $(\alpha_0,\infty)$ is actually not needed. This is the reason why we do not need to assume any subcritical growth
condition on the exponent $p$ in the model nonlinearities $g_2,g_3$. 
Let us recall that, in  the autonomous setting, Theorem 4 in \cite{evwe} 
yields nontrivial  radially symmetric solutions of \eqref{eq:ge} 
for superlinear nonlinearities, so that their results hold
for the nonlinearity $g_3$, but not for  $g_1,g_2$. 

Theorem \ref{th:nD} admits generalizations to some nonautonomous radially symmetric
nonlinearities. In particular we can prove a nonautonomous version of this result
that applies to the nonlinearities
\begin{align}  
g_1(r,z) &= -\lambda(r) z + \frac{z}{s(r)+z^2},  \label{g1g2_nonautonomous1} \\
g_2(r,z) &= k(r)^2 z \pm Q(r) |z|^{p-2}z, \label{g1g2_nonautonomous2}
\end{align} 
under suitable assumptions on the coefficients 
$\lambda(r),\,s(r),\,k(r),\,Q(r)$, see Theorem \ref{thm:general} and 
the Corollaries \ref{cor:NLH_nonauto} and \ref{cor:NLH2_nonauto}.
Our results in this context extend  Theorem~4 in \cite{evwe} in several directions (see Remark 
\ref{comparison:evwe}).  
  
 \smallskip
  
The existence of non-radially symmetric solutions is 
clearly a more difficult topic and here we can give a partial 
positive answer in this direction, by  exploiting the argument 
developed in  \cite{evwe, evwe_dual, ev_orlicz, evwe_branch}, 
where the authors study  the equation
\begin{equation}\label{eq:NLH_dual}
-\Delta u-k^{2} u=f(x,u)\quad\text{in }\R^N,
\end{equation}
where $f(x,u)$ is a super-linear nonlinearity satisfying 
suitable hypotheses, that include, for example, a  
non-autonomous generalization of our model nonlinearity 
$g_{3}$, i.e. 
$$
f(x,u)=Q(x)|u|^{p-2}u, \quad\text{ with }\quad  Q(x)>0.
$$
Among other results, in \cite{evwe_dual} (Theorem 1.1 and Theorem 1.2) it is shown that if $Q$ is $\mathbb{Z}^N$-periodic or vanishing at  infinity then there exist nontrivial solutions of
\eqref{eq:NLH_dual} for $p$ satisfying 
$\frac{2(N+1)}{N-1}<p<\frac{2N}{N-2}$ when $N\geq 3$.

Our contribution to this issue is that the positivity
assumption on $Q$ may be replaced by a negativity assumption in order to make the dual variational approach
work, so that, using  Fourier transform we show that the main ideas from \cite{evwe_dual} may be modified in such a
way that their main results remain true for negative $Q$.
Our results read as follows.   
    
\begin{thm} \label{Thm:variant_evwe_thm1.1}
Let $N\geq 3, \frac{2(N+1)}{N-1}<p<\frac{2N}{N-2}$ and let $Q\in 
L^\infty(\R^N)$ be periodic and negative almost everywhere. Then the 
equation \eqref{eq:NLH_dual} has a nontrivial localized oscillating strong 
solution in $W^{2,q}(\R^N)\cap C^{1,\alpha}(\R^N)$ for all
$q\in [p,\infty),\alpha \in (0,1)$.
\end{thm}
   
\begin{thm} \label{Thm:variant_evwe_thm1.2}
Let $N\geq 3, \frac{2(N+1)}{N-1}<p<\frac{2N}{N-2}$ and let $Q\in 
L^\infty(\R^N)$ be negative almost everywhere with $Q(x)\to 0$ as $|x|
\to\infty$. Then the equation  \eqref{eq:NLH_dual}
has a sequence of pairs $\pm u_m$ of nontrivial localized oscillating 
strong solutions in $W^{2,q}(\R^N)\cap C^{1,\alpha}(\R^N)$ for all 
$q\in [p,\infty),\alpha \in (0,1)$ such that
$$
  \|u_m\|_{L^p(\R^N)} \to \infty \quad\text{as }m\to\infty.
$$
\end{thm}

Since the above results together with those from \cite{evwe_dual} provide some existence results for the
Nonlinear Helmholtz equation associated with the nonlinearity $g_2$ from \eqref{g1g2_nonautonomous2}, one is
lead to wonder whether similar results hold true for asymptotically linear nonlinearities like $g_1$ in
\eqref{g1g2_nonautonomous1}. Here, the dual variational framework does not seem to be convenient since even
the choice of the appropriate function spaces is not clear. A thorough discussion of such nonlinear Helmholtz
equations leading to existence results for nonradial solutions still remains to be done. 

\smallskip
 
Let us observe that there is a gap in the admissible range of
 exponent between Theorem \ref{th:nD} and Theorem \ref{Thm:variant_evwe_thm1.1}. 
Reading  Theorem~\ref{th:nD} one is naturally  lead to the conjecture that nontrivial
nonradial solutions in $L^p(\R^N)$ may be found regardless of any sign condition on $Q$ and for all exponents
$p>\frac{2N}{N-1}$. On the contrary, Theorem \ref{Thm:variant_evwe_thm1.1}
only holds for exponents $p>\frac{2(N+1)}{N-1}$, so that it is still an open question
whether or not nonradial $L^p$-solutions exist for $p\in (\frac{2N}{N-1},\frac{2(N+1)}{N-1}]$.
  
  \smallskip
  
The paper is organized as follows: In section~\ref{sec:radial} we present 
the proof of Theorem~\ref{th:nD} as well as a generalization to the radial 
nonautonomous case  (see Theorem \ref{thm:general} and Corollaries 
\ref{cor:NLH_nonauto}, \ref{cor:NLH2_nonauto}).   In  section 
\ref{sec:nonradial} we present the proofs of Theorem~\ref{Thm:variant_evwe_thm1.1} and 
Theorem~\ref{Thm:variant_evwe_thm1.2}. In section \ref{variational} we will discuss in detail the 
attempt to obtain a solution by approximating $\R^{N}$ by
bounded domains.

\section{Radial solutions} \label{sec:radial}
  
\subsection{The autonomous case}
Throughout this section we will suppose that \eqref{g:reg}, \eqref{g:odd}, \eqref{g:der}, \eqref{g:sign} hold
true. We will prove Theorem~\ref{th:nD} by providing a complete understanding of the initial value problem
\begin{equation}\label{eq:equa}
-u'' - \dfrac{N-1}{r}u' =g(u)\quad\text{in }(0,\infty),\qquad u(0)=\alpha,\; u'(0)=0 
\end{equation}
for $\alpha\in\R$ and $N\in\N$. Notice that our assumptions on $g$ require that 
there exists a $\delta>0$ such that
$$
  g(z)z>0\quad \forall \,z\in (-\delta,\delta).
$$
Such a positivity region is in fact almost necessary as the following  result shows.
  
\begin{prop}\label{prop:nece}
Assume that $g\in C(\R)$ satisfies $g(z)z<0$ for all $z\in\R$. 
Then there is no nontrivial localized solution and 
there is no nontrivial oscillating solution $u\in C^{2}(\R^{N})$ of \eqref{eq:ge}.
\end{prop}
\begin{proof}
  Assume that $u\in C^{2}(\R^{N})$ is a nontrivial localized or
  oscillating solution. Then it attains a positive local maximum or a negative local minimum in some point
  $x_0\in\R^N$. Hence we obtain 
  $$
    -\Delta u(x_0) u(x_0) = u(x_0) g(u(x_0))<0,
  $$
  a contradiction. 
\end{proof}
\begin{remark}
  In view of elliptic regularity theory the above result is also true for weak solutions $u\in
  H^1(\R^N)$ since these solutions coincide almost everywhere with classical solutions and decay to zero at
  infinity by Theorem~C.3 in \cite{Simon_Schr}. 
  Notice that  in case $N\geq 3$  we can deduce the non-existence of $H^{1}(\R^{N})$
  solutions from the fact that $g(z)z<0$ in $\R$ violates the necessary condition (1.3) in \cite{beli}, see
  section 2.2 in that paper. In the case $N=2$ the same follows from Remarque 1 in \cite{BerGalKav}.
\end{remark}
First we briefly address the one-dimensional initial value problem  
\begin{equation} \label{eq:radial1D}
-u'' = g(u)\quad\text{in }(0,\infty),\qquad u(0)=\alpha,\; u'(0)=0.
\end{equation}
In view of the oddness of $g$ it suffices to discuss the inital value problem for $\alpha\geq 0$. The
uniquely determined solution of the initial value problem will be denoted 
by $u_\alpha$ with maximal existence interval $(-T_\alpha,T_\alpha)$ for 
$T_\alpha\in (0,\infty]$. 
  
\begin{prop}\label{prop:1D}
Let $N=1$. Then the following holds:
\begin{itemize} 
\item[(i)] If $\alpha=\alpha_0\in \R$ then $u_\alpha\equiv \alpha_0$ and 
if $\alpha=0$ then $u_\alpha\equiv 0$.
\item[(ii)] If $\alpha>\alpha_0$ then $u_\alpha$ 
strictly increases to $+\infty$ on $(0,T_\alpha)$.
\item[(iii)] If $0<\alpha <\alpha_0$ then $u_\alpha$ is periodic and 
oscillating with
$\|u_\alpha\|_\infty=\alpha$. 
\end{itemize}
\end{prop}
\begin{proof}
Conclusion {\it (i)} immediately follows from \eqref{g:sign}. 
Then we  only have to prove {\it (ii)} and {\it (iii)}. For notational convenience we write $u,T$ instead of
  $u_{\alpha},T_\alpha$. In the situation of {\it (ii)} we set $\xi:= \sup \{s\in [0,T) :u''(s)> 0\}$.
From  $u(0)=\alpha>\alpha_0$ and \eqref{eq:radial1D} we get
$u''(0) = -g(u(0))=-g(\alpha)>0$ and thus $\xi\in (0,T]$. We even have $
\xi=T$, because otherwise
$$
u(\xi) = \alpha + \int_0^\xi \int_0^t u''(s)\,ds \,dt > \alpha > 
\alpha_0
$$
and thus $u''(\xi)>0$ in view of assumption \eqref{g:sign} and 
\eqref{eq:radial1D}. This, however, would
contradict that $\xi$ is the supremum, hence $\xi=T$. As a 
consequence, $u$ is strictly convex on $(0,T_\alpha)$ which
implies {\it (ii)}. 
  
In order to show {\it (iii)} we notice that \eqref{g:odd} implies that 
solutions are symmetric about critical points and antisymmetric about zeros. Therefore, it suffices to  show that $u$ decreases until it attains a zero.
By the choice of $\alpha\in (0,\alpha_0)$ we have $u''(0)<0$ so that $u$ 
decreases on a right neighbourhood
of $0$. Exploiting \eqref{g:sign} and \eqref{eq:radial1D} we deduce that 
$u''(s)$ is negative whenever 
$0<u(s)<u(0)<\alpha_0$. As a consequence, we obtain that $u$ 
decreases as long as it remains positive. 
Moreover, it cannot be positive on $[0,\infty)$ since this would imply, 
thanks $0\leq u(r)\leq\alpha<\alpha_0$ and the assumptions \eqref{g:der},\eqref{g:sign},
\[
u''(r) +c(r)u(r)=0,\quad \text{with }c(r):=\frac{g(u(r))}{u(r)} \geq c_{0}>0.
\]
Hence, Sturm's comparison theorem (p.2 in \cite{Swa_comparison_and_oscillation})
ensures that $u$ vanishes somewhere, so that it cannot be positive in $[0,+\infty)$, a contradiction. Hence,
$u$ attains a zero and the proof is finished.
\end{proof}

Next, we consider the initial value problem \eqref{eq:equa} in the 
higher dimensional case $N\geq 2$. Again, we
may restrict our attention to the case $\alpha\geq 0$ and we
will denote with  $G(z)$ the primitive of the function $g(s)$, such that
$G(0)=0$. The following result furnishes the study of the solution set which are needed in the proof of
Again, the uniquely determined solution of the initial value problem
\eqref{eq:equa} will be denoted by $u_\alpha$ with maximal existence interval $(-T_\alpha,T_\alpha)$. 
\begin{remark}
There are many contributions concerning \eqref{eq:ge} in dimension
$N=1$, mainly related to some resonance phenomena.  In this
context, some ``Landesman-Lazer'' type conditions, joint
with suitable hypotheses on the nonlinearity $g$, are assumed 
in oder to obtain  existence of bounded, periodic or oscillating solution,
eventually with arbitrarily large $L^{\infty}$ norm, by
taking advantage of the presence on a forcing term in the equation
(see \cite{sove, ve} and the references therein).
Here the situation is different, as we do not need any monotonicity
assumption on $g$, nor the knowledge of the asymptotic behavior at infinity of $g$ is important, as it is  in \cite{sove, ve}. 
Moreover, our solutions satisfy a uniform $L^{\infty}$ bound, so that
the phenomenon we are dealing with is actually different from the
resonant one.

\end{remark}
\begin{lem}\label{lem:nD}
Let $N\geq 2$. Then the following holds: 
\begin{itemize}
\item[(i)] If $\alpha=\alpha_0\in \R$ then $u_\alpha\equiv \alpha_0$
 and if $\alpha=0$ then $u_\alpha\equiv 0$.
  \item[(ii)] If $\alpha>\alpha_0$ then $u_\alpha$ strictly increases to $+\infty$ on $[0,T_\alpha)$.
  \item[(iii)] If $0<\alpha <\alpha_0$ then $u_\alpha$ is oscillating, localized and satisfies
  \begin{equation} \label{eq:W1infty_bounds}
      \|u_\alpha\|_{L^\infty(\R)} = |\alpha| \quad\text{and}\quad
      \|u_\alpha'\|_{L^\infty(\R)} \leq \sqrt{2G(\alpha)} 
  \end{equation}
  as well as
  \begin{equation} \label{eq:bounds_at_infinity}
      c_\alpha r^{(1-N)/2} \leq |u_\alpha(r)|+|u_\alpha'(r)|+|u_\alpha''(r)| \leq C_\alpha r^{(1-N)/2} 
      \quad\text{for }r\geq 1
  \end{equation}
  for some $c_\alpha,C_\alpha>0$ depending on the solution but not on $r$.
  \end{itemize}
\end{lem} 
\begin{proof}
The existence and uniqueness
of a twice continuously differentiable solution $u_\alpha:(-T_\alpha,T_\alpha)\to\R$ can be deduced from 
Theorem~1 and Theorem~2 in \cite{ReiWal_radial}. 
We write again $u,T$ in place of $u_\alpha,T_\alpha$. The proof of {\it 
(i)} is direct and assertion  {\it (ii)} follows similar to the one-dimensional 
case. Indeed, note that $u''(0)>0$ because of
\begin{align*}
Nu''(0) = \lim_{r\to 0^{+}}u''(r)+\tfrac{N-1}{r}u'(r)  =-g(u(0))=-
g(\alpha)>0.
\end{align*}
Then, letting $\xi:= \sup \{s\in (0,T):  u'(s)>0\}$,  it results
$\xi \in (0,T]$. Assuming by contradiction that $\xi<T$ and using that
$\alpha>\alpha_{0}$, from \eqref{g:sign} we obtain 
$$
\xi^{N-1}u'(\xi) = - \int_0^\xi t^{N-1}g(u(t))\,dt  >0
$$
which is impossible, i.e. $\xi=T$. Then,
\eqref{eq:equa},\eqref{g:sign} and the maximality of $T$ yield {\it (ii)}. 
The proof of {\it (iii)} is lengthy so that  it will be subdivided into four 
steps.

\textit{Step 1: $u$ decreases to a first zero.}\;  
For all $r>0$ such that $0<u<\alpha_0$ on $[0,r]$ we have 
$$
r^{N-1}u'(r) = -\int_{0}^{r} t^{N-1}g(u(t))\,dt< 0,
$$
showing that $u$ decreases as long as it remains positive,
as in the one-dimensional case. Moreover, the function $u$ can not 
remain positive on $[0,\infty)$ because otherwise $v(r):= r^{(N-1)/2}u(r)$ would be
  a positive solution of
\begin{equation} \label{eq:v}
v'' + c(r)v= 0\qquad\text{where } c(r)= 
\frac{g(u(r))}{u(r)} -\frac{(N-1)(N-3)}{4r^{2}}.
\end{equation}
As in the proof of Proposition \ref{prop:1D} we observe $c(r)\geq c_{0}
>0$ for sufficiently large $r$ so that Sturm's comparison theorem tells us 
that $v$ vanishes somewhere. This is a contradiction to the positivity of $u$ and thus $u$ attains a first 
zero.

\textit{Step 2: $u$ oscillates and satisfies \eqref{eq:W1infty_bounds}.}\;    
Let us first show that there are $0=r_0<r_1<r_2<r_3<\ldots$ such that all $r_{4j}$ are local maximizers, all
  $r_{4j+2}$ are local minimizers and all $r_{2j+1}$ are zeros of $u$. Moreover, we will find that all zeros
  or critical points of $u$ are elements of this sequence and
  \begin{align} \label{eq:oscillations}
    2G(u(r_0))>u'(r_1)^2>2G(u(r_2))>u'(r_3)^2>2G(u(r_4))>\ldots
  \end{align}
In order to prove this we consider the function 
\begin{equation}\label{eq:Z}
Z(r):= u'(r)^2 + 2G(u(r)).
\end{equation}
and we observe that $Z$ decreases as
\begin{align}\label{Z:dec}
Z'(r)= 2u'(r)(u''(r)+g(u(r)))= - \frac{2(N-1)}{r}u'(r)^2< 0.
\end{align}
The existence of a first zero $r_1>0=r_0$ of $u$ has been shown 
in   \textit{Step 1} and the strict monotonicity of $Z$
  implies $Z(r_1)<Z(r_0)$. Concerning the behaviour of $u$ on $[r_1,\infty)$ there are now three
alternatives:
  \begin{itemize}
	\item[(a)] $u$ decreases until it attains $-u(r_0)$
	\item[(b)] $u$ decreases on $[r_1,\infty)$ to some value $u_\infty\in [-u(r_0),0)$
	\item[(c)] $u$ decreases until it attains a critical point at some $r_2>r_1$ with $-u(r_0)<u(r_2)<0$.
  \end{itemize} 
Let us show that the cases (a) and (b) do not occur. Indeed, if 
there exists $r>r_0$ such that $u(r)=-u(r_0)$, then,
by \eqref{eq:Z} we deduce that 
$$
Z(r)\geq 2G(u(r))=2G(u(r_0))=Z(r_0)
$$ 
which is  forbidden by  \eqref{Z:dec}. 
Then, in particular \eqref{eq:W1infty_bounds} holds.
Hence, the case (a) is impossible. Let us now suppose that (b) holds. Then  $u_\infty$  has to be a
  stationary solution of \eqref{eq:equa} and thus $u_\infty=-\alpha_0=-u(r_0)$. But then 
$$
Z(r)\geq 2G(u(r))\to 2G(u_\infty)=2G(-u(r_0))=Z(r_0) \quad\text{as }r\to\infty
$$ 
which again  contradicts \eqref{Z:dec}. 
So the case (c) occurs and there
must be a critical point $r_2$ with
\[
2G(u(r_2)) = Z(r_2)<Z(r_1)=u'(r_1)^2<Z(r_0)=2G(u(r_0)),
\]
so that \eqref{eq:equa}, \eqref{g:sign} and \eqref{g:odd} yield
\[ 
0>u(r_2)>-u(r_0) \text{ and }u'(r_2)=0, u''(r_2)>0 .
\]
Hence, $r_2$ is a local minimizer. Using that $Z$ is decreasing we can 
now repeat  the argument to get a zero $r_3>r_2$, a local maximizer 
$r_4>r_3$, a zero $r_5>r_4$ and so on. By the strict monotonicity of $Z$ 
one obtains \eqref{eq:oscillations} and thus \eqref{eq:W1infty_bounds}. 
Notice that this reasoning also shows that there are no further zeros or
critical points. 

\textit{Step 3: $u$ is localized.} \; First we show   $u(r)\to
0$ as $r\to\infty$. Our proof is similar to the one of Lemma 4.1 in 
\cite{GuiZhou} and it will be presented for the convenience of the reader.  
Take the sequence of maximizers $\{r_{4j}\}$ and assume by 
contradiction that $u(r_{4j})\to z\in (0,\alpha_0)$. Then 
\eqref{eq:equa} and Ascoli-Arzel\`a Theorem imply that 
$u(\cdot+r_{4j})$ converges locally uniformly to the unique solution $w$ 
of \eqref{eq:radial1D} with $w(0)=z,w'(0)=0$.
Proposition  \ref{prop:1D}, {\it (iii)} implies that this solution $w$ is 
$T$-periodic with two zeroes at $T/4,3T/4$. As a consequence,
there exists  $\delta>0$ such that $|w'|^2\geq 2\delta$ on
$[T/4-2\delta,T/4+2\delta]$. Hence, for sufficiently large $j_0\in\N$ we 
have for $j\geq j_{0}$
\begin{align*}
u'(r_{4j}+r)^2 
\geq \delta 
\quad \text{for $r\in [T/4-\delta,T/4+\delta]$ and } \quad
r_{4(j+1)}-r_{4j} &\geq T-\delta .
\end{align*}   
From this we deduce for $j\geq j_{0}$
\begin{align}\label{eq:armo}
 u'(r)^2  &\geq \delta,   \qquad\text{for }r\in [r_{4j}+T/4-\delta,r_{4j}+T/4+\delta], 
\\
\label{r:in}
r_{4j} &\geq r_{4j_0}+(j-j_0)(T-\delta)   \qquad\text{for } j\geq j_{0}.
\end{align}
Then, for $k\geq j_0$ and $r> r_{4k}+T/4+\delta$ we may exploit 
\eqref{Z:dec} and  \eqref{eq:armo} to obtain
\begin{align*}
Z(r)
&=Z(0)-2(N-1) \int_0^r \frac{u'(t)^2}{t} \,dt 
\\
&  
\leq Z(0)-2(N-1)\sum_{j=j_0}^k \int_{r_{4j}+T/4-\delta}^{r_{4j}+T/4+\delta} \frac{u'(t)^2}{t}\,dt 
\\
&\leq Z(0)- 2(N-1)\delta^2 \sum_{j=j_0}^k \int_{r_{4j}+T/4-\delta}^{r_{4j}+T/4+\delta}  \frac{1}{t} \,dt 
\\
&= Z(0)- 2(N-1)\delta^2 \sum_{j=j_0}^k
\ln\left(\dfrac{r_{4j}+\frac{T}4+\delta}{r_{4j}+\frac{T}4-\delta}\right). 
\end{align*}
Let us fix  $c(\delta)>0$  such that  
$\ln(1+x)\geq c(\delta) x$  for $0\leq x\leq 2\delta/(r_{4j_0} +\frac{T}4-\delta)$. Then \eqref{r:in} implies 
 \begin{align*}
Z(r) 
&\leq Z(0)- 2(N-1)\delta^2 c(\delta) \sum_{j=j_0}^k \dfrac{2\delta}{r_{4j} +\frac{T}4-\delta} 
\\
&\leq Z(0)- 2(N-1)\delta^2 c(\delta)     
\sum_{j=j_0}^k \dfrac{2\delta}{r_{4j_0} + (j-j_0)(T-\delta) +\frac{T}4-\delta}. 
\end{align*}	
Choosing now $k,r$ sufficiently large we obtain that  $Z(r)\to-\infty$
 because the harmonic series diverges, but \eqref{eq:W1infty_bounds}implies that $Z(r)\geq 2G(u(r))\geq 0$, yielding a contradiciton.
As a consequence, $u(r_{4j})$ converges to zero as $j\to\infty$ and 
analogously we deduce that also $u(r_{4j+2})\to 0$. In the end, we 
obtain $u(r)\to 0$ as $r\to + \infty$.
\\
Since $Z$ is decreasing and nonnegative it follows that $Z(r)\to Z_\infty\in [0,Z(0))$ as $r\to\infty$.
Hence, by \eqref{eq:Z}, also $|u'|$ has a limit at infinity which 
must be zero because $u$ converges to 0.
    Finally, from the differential equation we deduce that $u''(r)\to 0$ as $r\to\infty$, i.e.
\begin{equation} \label{eq:asymptotics}
u(r),u'(r),u''(r)\to 0 \quad (r\to\infty).
\end{equation}
As in Lemma 4.2 in \cite{GuiZhou} we get that for any $\eps>0$  there exists $C_{\eps}>0$
such that
\begin{equation} \label{eq:asymptoticsII}
|u(r)|,|u'(r)|,|u''(r)|\leq C_\eps r^{\frac{1-N}{2}+\eps} 
\qquad (r\geq 1)  .
\end{equation}

\textit{Step 4: Proof of \eqref{eq:bounds_at_infinity}.}\; Slightly generalizing the approach from
the proof of Theorem~4 in \cite{evwe} we study the function 
\begin{equation}\label{def:psi}
\psi(r):= v'(r)^2+ 2r^{N-1}G(u(r)) , \quad \text{where }
\;
v(r)=r^{(N-1)/2}u.
\end{equation}
Using the function $c$ from \eqref{eq:v} and taking into account \eqref{eq:equa}, 
we obtain that $\psi$ satisfies the following differential equation
\begin{align*} 
\psi'(r)
&= 2 v'(r)\left[-c(r)v(r) \right] + 2r^{(N-1)/2}g(u(r)) \left[v'(r)-
\frac{N-1}{2}r^{(N-3)/2}u(r)\right] 
\\
&\quad + 2(N-1)r^{N-2}G(u(r))   
\\
&=   (N-1)r^{N-2}\left(2G(u(r))- u(r)g(u(r))\right)   + \frac{ (N-1)(N-3)}{2r^2} v(r)v'(r).
\end{align*}
Taking into account \eqref{eq:asymptotics} and using \eqref{g:der} and \eqref{g:sign} we obtain that there exist
 $C,\,r_{0}\in (0,+\infty)$ such that
\begin{equation}\label{ineq:G}
\dfrac{u(r)^2}{G(u(r))}\leq  C\qquad \forall\,r\geq r_{0}.
\end{equation}
Then, exploiting \eqref{g:reg} and
\eqref{eq:asymptoticsII}, we find  
positive numbers $C',C'',r^*$ such that,
for  all $r\geq r^*$, it results
\begin{align*}
\begin{aligned}
 &\hspace{-1cm} \big| (N-1)r^{N-2}\left(2G(u(r))- u(r)g(u(r))\right) \big| \\  
 &
\leq \dfrac{(N-1)C}{2r} \frac{|2G(u(r))-u(r)g(u(r))|}{u(r)^2} \cdot 2r^{N-1}G(u(r)) \\
 &\leq \frac{C'}{r}|u(r)|^{\sigma} \psi(r) \\  
 &\leq C'' r^{-1+(\frac{1-N}{2}+\eps)\sigma} \psi(r).
\end{aligned}
\end{align*} 
Moreover, using \eqref{def:psi} and \eqref{ineq:G}, we get
\begin{align*}
\Big|  \frac{ (N-1)(N-3)}{2r^2}  v(r)v'(r) \Big| 
&\leq \frac{|(N-1)(N-3)|}{r^2}\cdot (v(r)^2+v'(r)^2) \\ 
&\leq \frac{|(N-1)(N-3)}{r^2} \cdot (Cr^{N-1}G(u(r))+v'(r)^2) \\
&\leq \frac{|(N-1)(N-3)|(C+1)}{r^2} \cdot \psi(r).
\end{align*}
This yields $\left|\psi'(r) \right|\leq a(r)\psi(r)$ for $r\geq r^*$ and some positive integrable function
$a$.
Dividing this inequality by the positive function $\psi(r)$ and integrating the resulting inequality over
	$[r^*,\infty)$ shows that $\psi$ is bounded from below and from above by a positive number. From this we
	obtain the lower and upper bounds \eqref{eq:bounds_at_infinity} and the proof is finished.
\end{proof}
        
We are now ready to give the proof of Theorem \ref{th:nD}.
       
\textit{Proof of Theorem~\ref{th:nD}}\; 
Let us define the set
$$
\mathcal{C} = \{ u_\alpha(|\cdot|)\in C^2(\R^N) : |\alpha|<\alpha_0\}$$ 
where $u_\alpha$ denotes the unique solution of the initial value problem 
\eqref{eq:equa}.
The set $\mathcal{C}$ is a subset of  $C^2(\R^N)$, and it is a 
continuum thanks to the Ascoli-Arzel\`a Theorem.
From Lemma~\ref{lem:nD} we obtain that all elements of 
$\mathcal{C}$ are oscillating localized solutions satisfying  \eqref{eq:W1infty_bounds} and
\eqref{eq:bounds_at_infinity}. \qed

\begin{remark}   
Let us mention that an analogous result to Theorem \ref{th:nD}
in Theorem 1 \cite{GuiLuo_II} and it is applied 
to a more restrictive class of nonlinearities.
Moreover, the above theorem is related to Theorem 4 in 
\cite{evwe} but we do not need their assumption $(g2)$. Actually, this hypothesis is not satisfied
in our model cases $g=g_1$ or $g=g_2$.
\end{remark}

\begin{remark}
The arguments from the proof of Theorem \ref{th:nD} also show the existence of oscillating localized 
solutions to initial value problems which are not of nonlinear Helmholtz type. For instance, one can 
treat concave-convex problems such as
\begin{equation} \label{eq:ABC}
-\Delta u  = \lambda |u|^{q-2}u + \mu |u|^{p-2}u \quad \text{in }\R^N,
\end{equation}
for $1<q<2<p<\infty$ with $\lambda>0,\mu\in\R$, see for instance \cite{ABC_combined_effects} or 
\cite{bep} for corresponding results on a bounded domain with homogeneous Dirichlet boundary conditions.
The existence of solutions is provided by Theorem 1 in \cite{ReiWal_radial} so that the steps 1,2,3  
are proven in the same way as above and we obtain infinitely many radially symmetric, oscillating, localized, solutions of
\eqref{eq:ABC}.
\end{remark}

\begin{remark}
  Using nonlinear oscillation theorems instead of Sturm's comparison theorem we can even extend
  the above observation towards superlinear nonlinearities $g$ satisfying $\lambda |z|^q\leq g(z)z\leq \Lambda |z|^q$ for 
  $$
      2<q\leq  \frac{2(N+1)}{N-1},N\in\{1,2,3\}
      \quad\text{or}\quad 
      2<q\leq \frac{2(N-1)}{N-2},N\geq 4.
  $$
  Indeed, in the first case the function $c$ from \eqref{eq:v} satisfies the estimate $c(r)\geq
  r^{(1-N)(q-2)/2}|v(r)|^{q-2}$ so that Atkinson's oscillation criterion applies, see the first line and
  third column of the table on p.153 in \cite{Swa_semilinear}. In the second case Noussair's oscillation
  criterion result can be used in order finish step~1,  see the third line and
  third column of the table on p.153 in \cite{Swa_semilinear}.
 
  \end{remark}

  \begin{remark}
    If $z\mapsto g(z)/z$ is decreasing, then one can
    show that the first zero of $u_\alpha$ is smaller than the first zero of $u_{\tilde\alpha}$ whenever
    $0<\alpha<\tilde\alpha<\alpha_0$. Indeed, we set $u:=u_\alpha, v:=u_{\tilde \alpha}$. Then the interval
     $$
      I:= \{t>0: u(s)>v(s)>0 \text{ for all }s\in (0,t)\}
    $$ 
    is open, connected and nonempty and thus $I=(0,r^*)$ for some $r^*>0$. On its right boundary we either
    have $u(r^*)=v(r^*)\geq 0$ or $u(r^*)>v(r^*)=0$; so it remains to exclude the first possibility. Using $u>v>0$
    on $I$ and \eqref{eq:equa} we have
    \begin{equation}\label{eq:1}
      \big(r^{N-1}(u'v-v'u)\big)'
      = r^{N-1}uv\Big(\frac{g(v)}{v}-\frac{g(u)}{u}\Big) > 0\quad\text{on }I.
	\end{equation}
	Integrating \eqref{eq:1} from $0$ to $r^*$ the assumption $u(r^*)=v(r^*)>0$ leads to
    $$
      0 < (u'v-v'u)(r^*) = u(r^*)(u-v)'(r^*),\quad \text{hence} \quad(u-v)'(r^*)> 0.
    $$
    On the other hand $u-v>0$ on $I=(0,r^*)$ and $(u-v)(r^*)=0$ implies $(u-v)'(r^*)\leq 0$, a contradiction.
    Thus $u(r^*)>v(r^*)=0$ so that the first zero of $v$ comes before the first zero of $u$. 
  \end{remark}
%

 \subsection{The nonautonomous case}
  
  In this section we generalize Theorem~\ref{th:nD} to a nonautonomous setting.
  Our aim is to identify mild assumptions on a nonautonomous nonlinearity $g$ that ensure the existence
  of a continuum of oscillating localized solutions of the initial value problems
  \begin{equation}\label{eq:radial_nonautonomous}
    -u'' - \dfrac{N-1}{r} u = g(r,u),\qquad u(0)=\alpha, \; u'(0)=0
  \end{equation}
  that behave like $r^{(1-N)/2}$ at infinity in the sense of \eqref{eq:bounds_at_infinity}. 
  Before formulating such assumptions and stating the corresponding existence result let us mention that our
  result applies to the nonlinearities \eqref{g1g2_nonautonomous1},\eqref{g1g2_nonautonomous2}
  under suitable conditions on the coefficient functions. This will be seen in
  Corollary~\ref{cor:NLH_nonauto} and Corollary~\ref{cor:NLH2_nonauto} at the end of this section. Our existence results for
  \eqref{eq:radial_nonautonomous} will be proven assuming that
  \begin{equation}
    \label{ggen:reg}  \text{$g\in C([0,+\infty)\times \R,\R)$ is continuously differentiable w.r.t. $r$.}
  \end{equation}
  Moreover, we suppose that there exist positive numbers $\alpha_*,\alpha^*,\lambda,\Lambda$
  and a locally Lipschitz continuous function $g_\infty:\R\to\R$ such that 
  \begin{align}
 	\label{ggen:lim}
	\lim_{r\to\infty}g(r,\cdot)&= g_\infty(\cdot) &&
	\text{uniformly on $[-\alpha_*,\alpha^*]$} 	\\
	\label{ggen:sign} g_r(r,z)z&\leq 0&&\text{on }[0,+\infty)\times [-\alpha_*,\alpha^*], \\
	\label{ggen:growth} \lambda z^2 \leq g_\infty(z)z&\leq g(r,z)z \leq \Lambda z^2
       &&\text{on }[0,+\infty)\times [-\alpha_*,\alpha^*].
  \end{align}
  These assumptions will allow us to prove the mere existence of an oscillating localized solution.
  In order to show the desired asymptotic behaviour we need some extra condition ''at infinity''
  where $r$ is large and the solution itself is small: We will assume that there exist $\eps,\sigma,C>0$
  and some integrable function $k$ such that
  \begin{align}
	\label{ggenasy1}|2G(r,z)- zg(r,z)| &\leq Cz^2 |\ln(z)|^{-1-\sigma},\qquad |z|\leq \eps,\, r\geq \eps^{-1} \\
	\label{ggenasy2}g_r(r,z)z &\geq -k(r) z^2,\quad\qquad\qquad |z|\leq \eps,r\geq \eps^{-1}.
  \end{align}
  These assumptions are rather technical but can be verified easily in concrete situations as we show in the
  proof of Corollary~\ref{cor:NLH_nonauto}. Let us remark that
  our assumptions \eqref{g:reg},\eqref{g:odd},\eqref{g:der}\eqref{g:sign} from the
  autonomous case (for any choice $\alpha^*=\alpha_*\in (|\alpha|,\alpha_0)$) are more restrictive than the
  assumptions used above. In particular, the following theorem generalizes our autonomous result. 
  
  \begin{thm} \label{thm:general} 
Let $N\geq 2$. Moreover assume  \eqref{ggen:lim},\eqref{ggen:sign},\eqref{ggen:growth} as
  well as
\begin{equation} \label{eq:choice_alpha}
G(0,\alpha)\leq \min\{ G_\infty(-\alpha_*),G_\infty(\alpha^*)\}
\quad\text{for }\alpha\in [-\alpha_*,\alpha^*].
\end{equation}
Then there is an oscillating, localized solution $u$ of \eqref{eq:radial_nonautonomous} that satisfies
$u(0)=\alpha$ as well as
    \begin{equation} \label{eq:W1infty_bounds_generalcase}
      \|u\|_{L^\infty(\R)} \leq \max\{-\alpha_*,\alpha^*\}\quad\text{and}\quad
      \|u'\|_{L^\infty(\R)} \leq \sqrt{2G(0,\alpha)}. 
    \end{equation}
Moreover, if \eqref{ggenasy1} and \eqref{ggenasy2} hold, 
then we can find $c,C>0$ such that 
\begin{equation} \label{eq:bounds_at_infinity_generalcase}
      c r^{(1-N)/2} \leq |u(r)|+|u'(r)|+|u''(r)| \leq C r^{(1-N)/2} 
      \quad\text{for }r\geq 1.
    \end{equation}
  \end{thm}
  \begin{proof}
    The proof of our result follows the same argument of the proof of Theorem~\ref{th:nD}, so we only mention the
    main differences. For simplicity  we only treat the case $\alpha>0$ with \eqref{eq:choice_alpha}.
    The existence of a maximally extended solution of \eqref{eq:radial_nonautonomous} follows from a Peano
    type existence theorem for singular initial value problems, see Theorem~1 in~\cite{ReiWal_radial}.

    \smallskip
    
    Step 1 is proven as in the autonomous case where the function $c$ from \eqref{eq:v} has to be replaced by 
    $c(r)= g(r,u(r))/u(r) - (N-1)(N-3)/4r^2$. Assumption \eqref{ggen:growth} ensures that $c$ is bounded from below by
    a positive constant as long as $0\leq u(r)<\alpha$ so that $u$ has to attain a first zero. In step 2 one
    shows that $Z(r):=u'(r)^2 + 2G(r,u(r))$ is nondecreasing due to $G_r(r,u(r))\leq 0$ for $-\alpha_*\leq
    u(r)\leq \alpha^*$, see \eqref{ggen:sign}. Arguing as in the autonomous case we find that $u$ decreases
    until it attains a local minimum at some $r_2>r_1$ with $-\alpha_*<u(r_2)<0$. More precisely one finds a sequence $(r_j)$
    such that all $r_{2j}$ are critical points and all $r_{2j+1}$ are zeros of $u$ with the
    additional property (the counterpart to \eqref{eq:oscillations}) 
    $$
	  2G(r_0,u(r_0)) >u'(r_1)^2 > 2G(r_2,u(r_2))> u'(r_3)^2 > 2G(r_4,u(r_4))>\ldots.
	$$
	This and $G(r_0,u(r_0))=G(0,\alpha)$ yields the $L^\infty$-bounds for $u'$ whereas
	the $L^\infty$-bounds follow from $-\alpha_*<u(r_{2j})<\alpha^*$ for all $j\in\N$. Hence, \eqref{eq:W1infty_bounds_generalcase} is
	proved so that step~2 is finished. Step~3 is the same as in the proof of Theorem~\ref{th:nD}. Since the
  	reasoning of Lemma 4.2 in \cite{GuiZhou} may be adapted to our nonautonomous (but asymptotically
  	autonomous) problem we also find \eqref{eq:asymptoticsII}, i.e. 
  	\begin{equation}\label{asy:nonauto} 
      |u(r)|,|u'(r)|,|u''(r)|\leq C_\eps r^{\frac{1-N}{2}+\eps} \qquad (r\geq 1)  
    \end{equation}
  	In step 4 we use \eqref{ggenasy1},\eqref{ggenasy2} in order to study	the asymptotics of the function 
$$ 
\psi(r):= v'(r)^2+ 2r^{N-1}G(r,u(r)) 
$$ 
where $v(r):=r^{(N-1)/2}u(r)$. One shows
\begin{align*} 
\psi'(r)	  
&= 	  2r^{N-1}G(r,u(r))\Big( \frac{N-1}{r}
\frac{2G(r,u(r))-  u(r)g(r,u(r))}{2G(r,u(r))}  + \frac{ G_r(r,u(r))}{G(r,u(r))}\Big)
\\
&\quad + \frac{ (N-1)(N-3)}{2r^2} v(r)v'(r).   
\end{align*}
For sufficiently large $r\geq r^*$ we get estimates $|G_r(r,u(r))|\leq k(r) 
G(r,u(r))$. Moreover, with an analogous inequality as in \eqref{ineq:G} 
as well as \eqref{ggenasy1},\eqref{asy:nonauto}  we get
\begin{align} \label{eq:est_asymptotics}
\begin{aligned}
\frac{2G(r,u(r))-u(r)g(r,u(r))}{2G(r,u(r))}
	  &\leq \frac{2G(r,u(r))-u(r)g(r,u(r))}{\lambda u(r)^2} \\
	  &\leq C|\ln(u(r))|^{-1-\sigma} \\ 
	  &\leq C'\ln(r)^{-1-\sigma}
	\end{aligned}
	\end{align}
	so that we may find as in the autonomous case a positive integrable function $a$ such that
	$|\psi'(r)|\leq a(r)\psi(r)$.
	This shows that $\psi$ is bounded from below and from above by a positive number.  
From this and 
\begin{equation} \label{eq:estimateG(r,z)}
\lambda	z^2\leq 2 G_\infty(z)\leq 2G(r,z)\leq \Lambda z^2 \quad\text{on
}[0,+\infty)\times [-\alpha_*,\alpha^*],
    \end{equation} 
    which is a consequence of \eqref{ggen:growth}, we obtain the lower and upper bounds
	\eqref{eq:bounds_at_infinity_generalcase} and the proof is finished.
  \end{proof}
  
  Finally let us apply Theorem~\ref{thm:general} to the special nonlinearities $g_1,g_2$ given in
  \eqref{g1g2_nonautonomous1},\eqref{g1g2_nonautonomous2}. We obtain the following results.

  \begin{cor} \label{cor:NLH_nonauto}
     Let $N\geq 2,p>2$ and suppose that $k,Q\in C^1([0,+\infty),\R)$ are nonincreasing
     functions with limits $k_\infty>0$ and $Q_\infty\in\R$, respectively. Then
     there is a nonempty open interval $I$ containing 0 and a continuum $\mathcal{C}= \{ u_\alpha \in
     C^2(\R^N): \alpha\in I \}\subseteq \,C^2(\R^N)$ consisting of radially symmetric oscillating classical
     solutions of the equation 
    $$
       -\Delta u - k(|x|)^2 u = Q(|x|) |u|^{p-2}u \quad \text{in }\R^N
    $$
    having the properties \eqref{eq:W1infty_bounds_generalcase},\eqref{eq:bounds_at_infinity_generalcase}
    stated in Theorem~\ref{thm:general}. In case $Q_\infty\geq 0$ we have $I=\R$.
  \end{cor}
\begin{proof}
We set
$$
g(r,z)= k(r)^2 z + Q(r)|z|^{p-2}z,\quad
g_\infty(z)= k_\infty^2 z + Q_\infty|z|^{p-2}z
$$
and  
$$
      G(r,z) = \frac{k(r)^2}{2} z^2 + \frac{Q(r)}{p} |z|^p, \quad
      G_\infty(z) = \frac{k_\infty^2}{2} z^2 + \frac{Q_\infty}{p} |z|^p.
    $$
    By the regularity assumptions on $k,Q$ we have \eqref{ggen:reg}. Moreover, $k',Q'\leq
    0$ implies $g_r(r,z)z\leq 0$ for all $z\in\R$ and thus \eqref{ggen:sign}. We set
    \begin{align} \label{eq:NLH_interval_I}
      \begin{aligned}
      I&:= \{\alpha\in \R : G(0,\alpha)<\sup_\R G_\infty \} \\
      &= \Big\{\alpha\in \R : G(0,\alpha)< \big(\frac{1}{2}-\frac{1}{p}\big)k_\infty^2
      \Big(\frac{{k_\infty^2}_{~}}{(Q_\infty)^{-}}\Big)^{\frac{2}{p-2}}\Big\}.
      \end{aligned}  
    \end{align}
    Here, $(Q_\infty)^{-}=\max\{-Q_\infty,0\}$. For any given $\alpha\in I$ we can choose
    \begin{equation*} 
     0<\alpha_* = \alpha^* <\Big(\frac{|V_\infty|_{~}}{(Q_\infty)^{-}}\Big)^{\frac{1}{p-2}} 
     \;\text{s.t.}\; G_\infty(-\alpha_*)=G_\infty(\alpha^*)= G(0,\alpha).
    \end{equation*}
    For this choice of $\alpha_*,\alpha^*$  assumption 
\eqref{ggen:growth} holds. Finally,
\eqref{ggenasy1} follows from $2G(r,z)-zg(r,z)=O(|z|^p)$ as $z\to 0$ 
uniformly with respect to $r$ and \eqref{ggenasy2} holds because 
$g_r(r,z)z\geq (2k(r)k'(r)+Q'(r))z^2$ for $|z|\leq 1$. Hence, all 
assumptions of Theorem~\ref{thm:general} are satisfied and the 
existence of solutions of 	\eqref{eq:radial_nonautonomous} follows. 
Due to the  unique solvability of these initial value problems and the Theorem of  Ascoli-Arzel\`{a} they form a continuum in $C^2(\R)$ with respect to the 
$C^2-$convergence on compact sets. Finally we remark that \eqref{eq:NLH_interval_I} implies $I=\R$ whenever $Q_\infty\geq 0$.
  \end{proof}

\begin{remark}\label{comparison:evwe}
Theorem~\ref{thm:general} extends Theorem 4 in \cite{evwe}
in various directions. First of all, it provides  more qualitative information 
of the solutions  such as the $W^{1,\infty}$-bounds, the oscillating 
behaviour of the solutions and the lower bounds for their decay at
infinity. Additionally, we do not assume any global positivity assumption 
on $f$. Furthermore, our assumption \eqref{ggenasy1}
is not covered by the hypotheses in \cite{evwe}.
\end{remark}   

The following result can be proved similarly and we state it for completeness. 
\begin{cor} \label{cor:NLH2_nonauto}
   Let $N\geq 2$ and suppose that $\lambda,s\in C^1([0,+\infty),\R)$ are nondecreasing functions with
    limits $\lambda_\infty,s_\infty$, respectively, such that $s$ is positive and $\lambda_\infty<1/s_\infty$. Then there is a nonempty
     open interval $I$ containing 0 and a continuum $\mathcal{C}= \{ u_\alpha \in C^2(\R^N): \alpha\in I \}$
     in $C^2(\R^N)$ consisting of radially symmetric oscillating classical solutions of the equation 
     $$
       -\Delta u  + \lambda(|x|) u =  \frac{u}{s(|x|)+u^2} \quad \text{in }\R^N
     $$
     having the properties \eqref{eq:W1infty_bounds_generalcase},\eqref{eq:bounds_at_infinity_generalcase}
     from Theorem~\ref{thm:general}. In the case $\lambda_\infty\leq 0$ we have $I=\R$.
   \end{cor}

\section{Nonradial solutions} \label{sec:nonradial}

In this section we study equation \eqref{eq:NLH_dual} proving
Theorem~\ref{Thm:variant_evwe_thm1.1} and Theorem~\ref{Thm:variant_evwe_thm1.2}.
We will follow the argument introduced in \cite{evwe_dual} adapting their methods to our context.
First note that, up to rescaling, we may assume $k=1$ in the following. Let us introduce some notations in order to facilitate the reading. 
Let  $\Psi$ be the real part of the fundamental solution
of the Helmholtz equation $-\Delta-1$ on $\R^{N}$
(see for example (11) in \cite{evwe_dual} ).
Performing  the transformation 
$v=|Q|^{1/p'}|u|^{p-2}u$ for $\frac{1}{p}+\frac{1}{p'}=1$
our problem amounts to solving  
\begin{equation} \label{eq:dual}
|v|^{p'-2}v = - |Q|^{1/p}[\Psi\ast (|Q|^{1/p}v)]\quad\text{in }\R^N.
\end{equation}
Notice that the right hand side comes with a negative sign in contrast to \cite{evwe_dual}. This is because
we assume $Q$ to be negative so that $Q=-|Q|$.  
Let us introduce  the linear operators 
$\bold{R},\, \bold{K}_p:L^{p'}(\R^N)\to L^p(\R^N)$
defined by
\begin{equation}\label{def:RK}
\bold{R}(v)= \Psi\ast v,\qquad 
\bold{K}_p(v) = |Q|^{1/p} \bold{R}(|Q|^{1/p}v).
\end{equation}
Both $\bold{R}$ and  $\bold{K}_p$ are continuous and we have
for all $f,g\in L^{p'}(\R^N)$
\begin{equation} \label{eq:R}
\int_{\R^N} f \bold{R}(g)
= \int_{\R^N} f (\Psi\ast g)
= \lim_{\eps\to 0} \int_{\R^N} \frac{(|\xi|^2-1)\hat f(\xi)\hat g(\xi)}{(|\xi|^2-1)^2+\eps^2} \,d\xi,
\end{equation}
where $\hat f,\hat g$ are the Fourier transforms of $f$ and $g$, 
respectively. In view of the variational structure of \eqref{eq:dual} we 
define the functionals $J,\,\bar J:L^{p'}(\R^N)\to \R$  via  the formulas
\begin{align} \label{eq:defJ}
\begin{aligned}
J(v) &:= \frac{1}{p'} \int_{\R^N} |v|^{p'} - \frac{1}{2} \int_{\R^N} v 
\bold{K}_p(v) \\
\bar J(v) 
&
:= \frac{1}{p'} \int_{\R^N} |v|^{p'} + \frac{1}{2} 
\int_{\R^N} v \bold{K}_p(v)
\end{aligned} 
\end{align}
so that the solutions  of \eqref{eq:dual} are precisely the critical points of $\bar J$, see (49) in
\cite{evwe_dual}. Notice that the functional $J$ is used when $Q$ is positive. Our main observation is that
not only $J$ but also $\bar J$ has the mountain pass geometry. This follows from the following Lemmas which
are the counterparts of Lemma~4.2 and Lemma~5.2 in \cite{evwe_dual}. In the following we will denote with
$\|\cdot\|_{q}$ the standard norm in the Lebesgue space $L^{q}(\R^{N})$.
  
  \medskip

  \begin{lem}\label{Lem:variant_evwe_lemma4.2ii}
    Under the assumptions of Theorem~\ref{Thm:variant_evwe_thm1.1} there is a function $v_0\in L^{p'}(\R^N)$
    such that $\|v_0\|_{p'}>1,\bar J(v_0)<0$.
  \end{lem}
  \begin{proof}
As in Lemma 4.2(ii) \cite{evwe_dual} it suffices to prove 
\begin{equation}
\int_{\R^N} z \bold{K}_p z< 0
\end{equation} 
for some $z\in L^{p'}(\R^N)$ because then one may take $v_0:= t z$ for sufficiently large $|t|$. To this end  let $y\in \mathcal{S}(\R^N)$ be a nontrivial Schwartz function satisfying $ \supp(\hat y) \subset B_1(0)$. For $\delta>0$ we set
$$
  z_\delta:= y|Q|^{-1/p} 1_{\{|Q|>\delta\}},\qquad 
  \mu := \int_{\supp(\hat y)} \frac{|\hat y(\xi)|^2}{|\xi|^2-1}\,d\xi < 0,
$$
where $1_{\{|Q|>\delta\}}$ is the indicator function of the set 
$\{x\in \R^{N}\,:\,|Q(x)|>\delta\}$.
    Then we have $z_\delta\in L^{p'}(\R^N)$ and thus $\bold{K}_pz_\delta \in L^p(\R^N)$. Hence, by
    definition of $\bold{K}_p$, the function $y_\delta:= |Q|^{1/p}z_\delta = y\cdot 1_{\{|Q|>\delta\}}$ satisfies
    \begin{align*}
      \int_{\R^N} z_\delta (\bold{K}_p z_\delta)
      =  \int_{\R^N} z_\delta|Q|^{1/p} \bold{R}(|Q|^{1/p}z_\delta)
      =  \int_{\R^N} y_\delta \bold{R}(y_\delta). 
    \end{align*}
Since we have $|Q|>0$ almost everywhere, we get 
$y_\delta\to y$ in $L^{p'}(\R^N)$ as $\delta\to 0^{+}$. Thus the continuity of
    $\bold{R}$ implies that we can choose $\delta>0$ so small that the following holds:
    \begin{align*}
      \int_{\R^N} z_\delta (\bold{K}_p z_\delta)
      <  \int_{\R^N} y\bold{R}(y) + \frac{|\mu|}{2}.  
    \end{align*} 
    From this and \eqref{eq:R} we infer
    \begin{align*} 
      \int_{\R^N} z_\delta (\bold{K}_p z_\delta)
      <  \lim_{\eps\to 0}  \int_{\R^N} \frac{(|\xi|^2-1)|\hat y(\xi)|^2}{(|\xi|^2-1)^2+\eps^2}\,d\xi  + 
      \frac{|\mu|}{2}      
      =  \mu + \frac{|\mu|}{2} < 0
    \end{align*}
    which is all we had to show.
  \end{proof}
   
   \medskip
    
  \begin{lem}\label{Lem:variant_evwe_lemma5.1}
    Let the assumptions of Theorem~\ref{Thm:variant_evwe_thm1.2} hold. Then for every $m\in\N$ there is an
    $m$-dimensional subspace $\mathcal{W}\subset L^{p'}(\R^N)$ with the following properties:
    \begin{itemize}
      \item[(i)] $\dys \int_{\R^N} v\bold{K}_pv < 0$ for all $v\in \mathcal{W}\sm\{0\}$.
      \item[(ii)] There exists $R = R(\mathcal{W}) > 0$ such that $\bar J(v) \leq  0$ for every $v\in \mathcal{W}$ with
      $\|v\|_{p'}\geq R$.
    \end{itemize}
  \end{lem}
  \begin{proof}
    Let  $y^1,\ldots,y^m\in\mathcal{S}(\R^N)$ be nontrivial Schwartz functions such that 
    \begin{equation} \label{eq:lem_symmetric_MP}
      \bigcup_{j=1}^m \supp(\hat y^j) \subset B_1(0), \qquad
      \supp(\hat y^j)\cap \supp(\hat y^i) =\emptyset \quad (i\neq j). 
    \end{equation}
    For sufficiently small $\delta>0$ we then define 
    $$
      \mathcal{W}:=\spa\{z_\delta^1,\ldots,z_\delta^m\}\qquad\text{where }
      z_\delta^j :=  y^j|Q|^{-1/p} 1_{\{|Q|>\delta\}}.
    $$
    Then \eqref{eq:lem_symmetric_MP} implies that $\mathcal{W}$ is $m$-dimensional and similar
    calculations as above show (i) and (ii).
  \end{proof}
  
  With the aid of the above Lemmas the proofs of our theorems are essentially the same as in \cite{evwe}. We
  indicate the main steps for the convenience of the reader.
  
  \medskip
  
\textit{Proof of Theorem \ref{Thm:variant_evwe_thm1.1}:}\;   
Under the given assumptions $\bar J$ has the Mountain Pass 
geometry.  
Indeed, as in the parts (i),(iii) of Lemma~4.2 in \cite{evwe_dual} one 
proves that $0$ is a strict local minimum and the boundedness of 
Palais-Smale sequences of $\bar J$. 
In  Lemma~\ref{Lem:variant_evwe_lemma4.2ii} we proved that there is 
a $v_0\in L^{p'}(\R^N)$ such that
$\|v_0\|_{p'}>1,\bar J(v_0)<0$. Hence, as in Lemma~6.1 
\cite{evwe_dual} the Deformation Lemma implies the
existence of a bounded Palais-Smale sequence $(v_{m})$ 
for $\bar J$ at its Mountain-Pass level $\bar c>0$. Similar to
the proof of Theorem~6.2 in \cite{evwe_dual} one has 
  \begin{align*}
    \lim_{m\to\infty} \int_{\R^N}|Q|^{1/p}\bold{R}(|Q|^{1/p}v_m)
    &= \frac{2p'}{2-p'}\lim_{m\to\infty} \Big[ - \bar J(v_m)+\frac{1}{p'}\bar J'(v_m)[v_m]\Big] \\
    &= - \frac{2p'}{2-p'} \bar c 
    < 0.
  \end{align*}
Then, Theorem~3.1 in \cite{evwe_dual} implies that there are $R,\zeta>0$ and points $x_m\in\R^N$ and a subsequence, still
denoted with $(v_m)$,  such that 
$$
\int_{B_R(x_m)} |v_m|^{p'} \geq \zeta>0.
$$
From this point on the reasoning is the same as in \cite{evwe_dual} and 
we obtain that $(v_m)$ converges weakly to a nontrivial solution $v$ 
of \eqref{eq:NLH_dual} and the solution $u$ of the original equation
 may be found via $u= \bold{R}(|Q|^{1/p}v)\in L^p(\R^N)$. In 
 particular $u$ satisfies \eqref{eq:NLH_dual} and
$$
\lim_{|x|\to\infty} \int_{|x-y|\leq 1} |u(y)|^p \,dy = 0
$$
so that replacing 2 by $p$ in the proof of Theorem C.3.1 in 
\cite{Simon_Schr} one proves $u(x)\to 0$ as
$|x|\to\infty$, namely that $u$ is localized. 
This implies 
$$
\Delta u + (k^2+o(1))u = 0,\quad \text{as } |x|\to\infty
$$ 
so that the  PDE version of  Sturm's comparison principle (for instance 
Theorem 5.1 in \cite{Swa_comparison_and_oscillation}) 
and the Strong Maximum Principle show that $u$ is oscillating. 
\qed  
    
\textit{Proof of Theorem \ref{Thm:variant_evwe_thm1.2}:}\; 
Lemma \ref{Lem:variant_evwe_lemma5.1} yields all
the required geometrical features of the symmetric Mountain Pass
Theorem (see Theorem~6.5 in \cite{struwe}). 
Moreover,  Lemma 5.2   in \cite{evwe_dual} implies that
the Palais-Smale condition holds for $\bar J$, giving the
existence of pairs of non-trivial localized solutions.
The oscillation property follows again from Theorem 5.1 in
\cite{Swa_comparison_and_oscillation}.

\qed
  
  \medskip 

\section{On the approximation by bounded domains}\label{variational}

In this section we briefly address the question whether localized 
solutions of \eqref{eq:ge} can be
approximated by solutions of the corresponding homogeneous Dirichlet 
problem on a large bounded domain. We are
going to show that this method does not work in general. More 
precisely, we will prove that the positive
minimizers of the  Euler functionals associated with the problem 
on bounded domains diverge in $H^1(\R^N)$ 
as the domains approach $\R^N$\!.
Even though the divergence will only be proved for the sequence of minimizers we believe that the analogous
phenomenon occurs for broader classes of finite energy solutions, e.g. constrained minimizers, or solutions with a given upper bound
on their nodal domains or on their Morse index.  Throughout this section we will assume that the nonlinearity
$g$ satisfies the hypotheses \eqref{g:reg},\eqref{g:odd},\eqref{g:der} as well as
\eqref{g:sign} with $\alpha_{0}\in (0,+\infty)$, in order to avoid some sub-critical growth conditions
(see Remark \ref{app:rem2}).

Let $\Omega\subset\R^N$ be a bounded domain, and consider the variational problem
\begin{equation}\label{defc}
  c_\Omega:= \inf_{H_0^1(\Omega)} I_{\Omega}
  \quad\text{where }  I_{\Omega}(u) = \frac{1}{2} \int_{\Omega}  |\nabla u|^2  - \int_\Omega G(u)
\end{equation}
where $G(z)$ denotes the primitive of $g$ such that $G(0)=0$. Notice that
\eqref{g:reg},\eqref{g:der},\eqref{g:sign} implies 
$G(z)\leq C|z|^2$ for some $C>0$ and for all
$z\in\R$, so that $I_\Omega:H_0^1(\Omega)\mapsto \R\cup\{+\infty\}$ is well-defined.
Bounded critical points of $I_\Omega$ are classical solutions of the boundary value problem
 \begin{equation}  \label{eq:bounded}
\begin{cases}  
-\Delta w  = g(w) &\text{in }\Omega, \\ 
 w\in H_0^1(\Omega). &
\end{cases}  
\end{equation}
In the following proposition we show that $I_\Omega$ admits a positive minimizer provided
$g'(0)>\lambda_1(\Omega)$ holds. More precisely, we have the following result.

\begin{prop} \label{app:Prop}
Let $\Omega\subset\R^N$ be a bounded domain and in addition to 
\eqref{g:reg},\eqref{g:odd},\eqref{g:der},\eqref{g:sign} assume $g'(0)>\lambda_1(\Omega)$. 
Then, there exists a global minimizer $u_\Omega$ of $I_\Omega$ in $H^{1}_{0}(\Omega)$ which is a solution of \eqref{eq:bounded}
  satisfying $0 < u_\Omega < \alpha_0$ in $\Omega$.
\end{prop}
\begin{proof}
Hypotheses \eqref{g:odd} and \eqref{g:sign} imply that $G(z)\leq G(\alpha_0)$ holds for every $z\in \R$.
Hence, for all $u\in H_0^1(\Omega)$ we have
$$
I_\Omega(u) \geq \frac{1}{2} \int_{\Omega}  |\nabla u|^2 - \int_\Omega G(\alpha_0)\dx = \frac{1}{2}
\int_{\Omega}  |\nabla u|^2  - |\Omega| G(\alpha_0), 
$$ 
which shows that $I_\Omega$ is coercive and bounded from below.
Moreover, if $\phi_1$ denotes the eigenfunction associated to $\lambda_1(\Omega)$, then  
\begin{align*}
\lim_{t\to 0} \frac{I_{\Omega}(t\phi_1)}{t^2}
= \frac{1}{2}\int_\Omega |\nabla \phi_1|^2 - G''(0)\phi_1^2
= \frac{1}{2}\Big(\lambda_1(\Omega)-g'(0)\Big) \int_\Omega \phi_1^2 < 0,
\end{align*}
so that $c_{\Omega}<0=I_{\Omega}(0)$. Additionally, $I_\Omega$ is weakly sequentially lower semicontinuous so
that there exists a minimizer $u_\Omega$, which must be nontrivial because of $c_\Omega <0$. 
We may assume $0\leq u_\Omega\leq \alpha_0$ because  
$\min\{|u_\Omega|,\alpha_0\}\in H_0^1(\Omega)$ is another minimizer of $I_\Omega$. From the strong
maximum principle we deduce that $u_\Omega$ satisfies $0<u_\Omega<\alpha_0$ in $\Omega$ as it is nontrivial.  
\end{proof}

\begin{remarks} \label{app:rem}
\begin{itemize}
\item[(a)] If  $z\mapsto g(z)/z$ is decreasing, then the condition $g'(0)>\lambda_1(\Omega)$ is even necessary for the existence of a positive solution $u\in
H_0^1(\Omega)$. Indeed, testing \eqref{eq:bounded} with $u$ gives $$
\lambda_1(\Omega)\, \int_\Omega  u^2
\leq \int_\Omega  |\nabla u|^2 
= \int_\Omega g(u)u
< g'(0) \,\int_\Omega u^2.     
$$ 
In particular, note that our model nonlinearities $g_1,g_2$ 
given in \eqref{model:g1} and \eqref{model:g23} satisfy this monotonicity property. 
\item[(b)] If $\Omega$ is smooth then Theorem~1 in \cite{bros} shows that the
the positive solution of \eqref{eq:bounded} is unique provided  $z\mapsto g(z)/z$ is decreasing. 
\end{itemize}
\end{remarks}
\begin{remark}\label{app:rem2} 
In this section we do not consider the case $\alpha_0=+\infty$ 
in \eqref{g:sign} because, without imposing additional growth
conditions,  the functional  $I_\Omega$ may  not be well-defined in 
this case and, even if it were, it need not be bounded from below.
\end{remark}
  
Next we study the convergence of the minimizers obtained in 
Proposition~\ref{app:Prop}. To this end, we consider a sequence $(\Omega_n)$ 
of bounded domains satisfying
$\Omega_n\subset\Omega_{n+1}\subset\R^N$ and $
\bigcup_{n\in\N} \Omega_n = \R^N$. Since every compact subset
of $\R^N$ is covered by finitely many of those bounded domains we 
observe that $\lambda_1(\Omega_n)\to 0$ as $n\to\infty$ so that, by the above proposition,
the existence of positive minimizers is guaranteed for
large $n$ provided that \eqref{g:der} holds. We show that the minimizers 
converge to the constant solution $\alpha_0$ and therefore do not 
give any new finite energy solution.

 \begin{thm}\label{teo:app}
Assume \eqref{g:reg},\eqref{g:odd},\eqref{g:der},\eqref{g:sign}
and  let $(\Omega_n)$ be a sequence of bounded domains such that 
$\Omega_{n}\subseteq \Omega_{n+1}$ and $\cup_{n}\Omega_{n}
=\R^{N}$. Then, for all sufficiently large $n$, there exists a nontrivial 
minimizer   $u_n$ of $I_{\Omega_n}$ on $H^1_0(\Omega_n)$ having 
the following properties:
\begin{itemize}
\item[(a)] $0<u_n<\alpha_0$ in $\Omega_n$, 
\item[(b)] $I_{\Omega_n}(u_{n})\to -\infty$,
\item[(c)] $u_n\to \alpha_0$ in $C^\infty_{loc}(\R^N)$ and $\|u_n\|
_{L^q(\Omega_n)}\to \infty$
for all $q\in [1,\infty)$.
\end{itemize}
\end{thm}
\begin{proof}
Since $\lambda_{1}(\Omega_{n})\to 0$, taking into account 
\eqref{g:der} we find $n_{0}$ such that, for every $n\geq n_{0}$,
$\lambda_{1}(\Omega_{n})<g'(0)$. As a consequence, we can apply
Proposition \ref{app:Prop} to deduce that there exists a sequence
$(u_{n})_{n\geq n_{0}}$ of positive minimizers of 
$I_{\Omega_{n}}$ satisfying conclusion {\it (a)}.
In order to prove conclusion {\it (b)}  let $\phi\in
C_0^\infty(\R^n)$ be given with $\|\phi\|_2=\|\phi\|_\infty=1$. 
For every $k\in \N$ we set 
$$
\phi_k(x) = \frac{1}{k^{N/2}}\phi(\tfrac{x}{k}+ke_1) ,
$$
so that $\|\phi_k\|_2=1,\|\phi_k\|_\infty\leq 1,\| \nabla \phi_k\|_2\to 0$. 
Without loss of generality, we may assume 
$\sigma\in (0,1)$ from hypothesis \eqref{g:reg} to be so small 
that $2+\sigma\in [2,\frac{2N}{N-2}]$ holds provided $N>2$. Exploiting \eqref{g:reg} and \eqref{g:sign} we
obtain positive constants $A,\,C$ such that
$$
g'(0)z^2-2G(z) \leq A |z|^{2+\sigma}\qquad\text{for }|z|\leq 1,\qquad       
\|\phi_k\|_{2+\sigma}\leq C.
$$
Then, for a fixed positive $t\leq
  \min\{(g'(0)/(4AC))^{1/\sigma},1\}$ and sufficiently large $k\geq k_0$ we have
  $\|t\phi_k\|_\infty\leq 1$ so that the following estimate holds  
  \begin{align*}
   2I(t\phi_k)
   = &t^2 \int_{\R^N} |\nabla \phi_k|^2-g'(0) \phi_k^2 
    + \int_{\R^N} \big( g'(0)(t\phi_k)^2 -  2G(t\phi_k)\big)   \\
   \leq &- \dfrac{g'(0)}2 t^2   + A \int_{\R^{N}} |t\phi_k|^{2+\sigma}  
   \leq 
   \dfrac{t^2}2\left(-g'(0) + 2t^\sigma AC\right) =:-E,
  \end{align*}
where $E>0$ by the  choice of $t$. 
Since the supports of $(\phi_k)$ go  off to infinity we find some 
$k_1\in\N$ such that for all $k\geq k_1$ it results
\begin{align*}
2I(t\phi_{k_0}+t\phi_k)
\leq \dfrac{E}{2} + 2I(t\phi_{k_0}) + 2I(t\phi_k) 
\leq - \frac{3}{2}E.
\end{align*}
Inductively, we find $k_2< k_3< \ldots$ such that for all $k\geq k_m$ 
we have 
\begin{align*}
2I(t\phi_{k_0}+t\phi_{k_1}+\ldots+t\phi_k)
&\leq \dfrac{E}2 + 2I(t\phi_{k_0}+t\phi_{k_1}+\ldots+t\phi_ 
{k_{m-1}}) + 2I(t\phi_k) \\ 
&\leq - (1+m/2)E. 
\end{align*}
Since for any given $m\in\N$ supp$(t\phi_{k_0}+t\phi_{k_1}+
\ldots+t\phi_k)\subset \Omega_{n}$ for sufficiently
large $n$, the same estimate holds true for $I_{\Omega_n}$,
 yielding conclusion {\it (b)}.

In order to show  {\it (c)}, note that the sequence $(u_{n})$
is made of minimizers of $I_{\Omega_{n}}$ so that
\begin{equation} \label{eq:un_stable}
\int_{\Omega_n} g'(u_n)\phi^2  \leq \int_{\Omega_n} |\nabla \phi|^2 \qquad\text{for all }\phi\in C_c^1(\Omega_n).
\end{equation}
By the Ascoli-Arzel\`{a} Theorem and interior Schauder estimates we 
find that $(u_n)$ converges in $C^2_{loc}(\R^N)$
 to some limit 
function $u\in C^2(\R^N)$ satisfying $0\leq u(x) \leq
\alpha_0$ for all $x\in\R^N$ as well as \eqref{eq:ge}.
The Dominated Convergence Theorem allows to pass to the limit in \eqref{eq:un_stable} and we obtain
\begin{equation*} 
\int_{\R^N} g'(u)\phi^2  \leq \int_{\R^N} |\nabla \phi|^2 \qquad\text{for all }\phi\in C_c^1(\R^N),
\end{equation*}
There are now three possibilities: either  $u\equiv 0$ or
$u$ is non-constant or $u\equiv \alpha_0$. As a consequence, it is left to show that the first two
possibilities do not occur.

First, assume by contraction that $u\equiv 0$, so that
for any given compact $K$ we have $u_n\to u$ uniformly on $K$. Thanks to
\eqref{g:der} we can find a sufficiently large $n$ such that $g(u_{n})/u_{n}\geq \delta^{2}:=g'(0)/2$. We
may choose $K$ so large such that the fundamental solution $\psi$ of 
$$
\Delta \psi+\delta^2\psi=0
$$
changes sign within $K$. Then the PDE version of Sturm's comparison theorem (Theorem 5.1 in  
\cite{Swa_comparison_and_oscillation}) shows that $u_n$ has a zero within $K$  
contradicting  the positivity of $u_{n}$. \\
Assume now that $u$  is non-constant. 
Arguing as in the proof of Theorem 1.3 in \cite{Farina_some_symmetry} 
and applying Proposition~1.4 in \cite{Farina_some_symmetry} 
one shows that $u>0$ and $\|u\|_{\infty}<\alpha_{0}$.
As a consequence,  the constant
$$
c_{0}:=\min_{0\leq s\leq \|u\|_\infty} \frac{g(s)}{s}
$$
turns out to be positive and,
for any compact set $K$ we can find $n$ such that 
 $g(u_n)/u_n\geq c_{0}/2$. Choosing again $K$ sufficiently large we get a contradiction as above.

Hence, it turns out that $u\equiv \alpha_{0}$ and in particular we get $\|u_n\|_{L^q(\Omega_n)}\to  \infty$
for all $q\in [1,\infty)$.
%
  \end{proof}


\bibliographystyle{alpha}
\bibliography{bibliography}

\end{document}